\documentclass{article}%
\usepackage{amsmath}
\usepackage{amsfonts}
\usepackage{amssymb}
\usepackage[colorlinks,linkcolor={blue},  bookmarks={false},
pdfpagelayout={SinglePage}  ]{hyperref}
\usepackage[UKenglish]{babel}
\usepackage[UKenglish]{isodate}
\usepackage{graphicx}
\usepackage{tikz}%
\numberwithin{equation}{section}
\newtheorem{theorem}{Theorem}[section]

\newtheorem{corollary}[theorem]{Corollary}

\newtheorem{definition}[theorem]{Definition}
\newtheorem{example}[theorem]{Example}

\newtheorem{remark}[theorem]{Remark}

\newenvironment{proof}[1][Proof]{\noindent\textbf{#1.} }{\ \rule{0.5em}{0.5em}}
\begin{document}

\title{Fundamental Theorem of Asset Pricing under fixed and proportional transaction costs}
\author{Martin Brown and Tomasz Zastawniak\thanks{Corresponding author; email:
\texttt{tomasz.zastawniak@york.ac.uk}.}\medskip\\{\small Department of Mathematics, University of York}\\{\small Heslington, York YO10~5DD, United Kingdom}}
\date{\cleanlookdateon\today}
\maketitle

\begin{abstract}
We show that the lack of arbitrage in a model with both fixed and proportional
transaction costs is equivalent to the existence of a family of absolutely
continuous single-step probability measures, together with an adapted process
with values between the bid-ask spreads that satisfies the martingale property
with respect to each of the measures. This extends Harrison and Pliska's
classical Fundamental Theorem of Asset Pricing to the case of combined fixed
and proportional transaction costs.

\end{abstract}

\section{Introduction}

The Fundamental Theorem of Asset Pricing, linking the lack of arbitrage with
the existence of a risk neutral probability measure, has been studied for a
diverse range of models of the financial market. The first to establish the
result for discrete time models with finite state space were Harrison and
Pliska \cite{HarPli1981}. Dalang, Morton and Willinger \cite{DalMorWil1990}
extended the theorem to the case of infinite state space, and Delbaen and
Schachermayer \cite{DelSch1994}, \cite{DelSch1998} to continuous-time models.

The above classical results apply to frictionless models. Harrison and
Pliska's result was extended to models with friction in the form of
proportional transaction costs (represented as bid-ask spreads) by Jouini and
Kallal \cite{JouKal1995}, Kabanov and Stricker \cite{KabStr2001} and Ortu
\cite{Ort2001}. Furthermore, Roux \cite{Rou2011} included interest rate
spreads in addition to proportional transaction costs. Similarly, the result
by Dalang et~al.\ involving an infinite state space was extended to models
with proportional transaction costs by Zhang and Deng \cite{ZanDen2002},
Kabanov, R\'{a}sonyi and Stricker \cite{KabRasStr2002}, and Schachermayer
\cite{Sch2004}.

Under fixed transaction costs, to our best knowledge, the equivalence between
the absence of arbitrage and the existence of risk neutral measures has so far
been studied in just one paper, by Jouini, Kallal and Napp
\cite{JouKalNap2001}.

The long-standing question of extending the Fundamental Theorem of Asset
Pricing to cover the situation when both fixed and proportional transaction
costs apply simultaneously is addressed in the present paper. We use the term
`combined costs' as shorthand when referring to this case. Such costs are
ubiquitous in the markets, hence it is important to be able to characterise
the lack of arbitrage in their presence. In Theorem~\ref{Thm:0sjfd67fbd} we
show that the absence of arbitrage in a market with combined costs is
equivalent to the existence of a family of single-step probability measures
absolutely continuous with respect to (but not necessarily equivalent to) the
physical probability, along with a martingale with respect to such a family of
measures (as defined in Section~\ref{Sect:mfd61oa0}) and taking values between
the bid and ask prices. In doing so, we extend the classical result of
Harrison and Pliska \cite{HarPli1981} for a finite state space to the case of
combined costs. Later on, in Corollary~\ref{Cor:f8f56sgfsmac} we provide
another equivalent condition for the lack of combined-cost arbitrage, namely
the existence of an embedded arbitrage-free model with fixed costs.

The technical difficulties inherent in the problem solved here are due to a
combination of two factors. On the one hand, proportional costs mean that the
absence of arbitrage in the full multi-step model is not equivalent to the
condition that every single-step submodel should be arbitrage free (even
though such an equivalence holds in frictionless models as well as under fixed
costs), preventing an argument by reduction to a single step. On the other
hand, fixed costs imply that the set of solvent portfolios lacks convexity.
While these difficulties have been tackled separately in the context of
proportional costs and, respectively, fixed costs only, they require fresh
ideas to handle their compounded effect. This is achieved in the proof of
Theorem~\ref{Thm:0sjfd67fbd}.

Finally, we mention the recent paper by Lepinette and Tran \cite{LepTra2017},
in which arbitrage under market friction involving lack of convexity (and
including the case of simultaneous fixed and proportional costs) has been
considered. In that paper the absence of asymptotic arbitrage is characterised
by the existence of a so-called equivalent separating probability measure.
However, no link is made with risk neutral probabilities, by contrast to the
present paper. In Example~\ref{Exl:nf7am0ax} we show that the non-existence of
an equivalent separating probability measure does not, in fact, mean that a
combined-cost arbitrage opportunity must be present.

\section{Notation and preliminaries\label{Sect:mfd61oa0}}

Let $T$ be a positive integer and let $\left(  \Omega,\Sigma,\mathbb{P}%
\right)  $ be a finite probability space equipped with a filtration
$\mathcal{F}=\left(  \mathcal{F}_{t}\right)  _{t=0}^{T}$. We assume (without
loss of generality) that the \textbf{physical measure}~$\mathbb{P}$ satisfies
the condition $\mathbb{P}(A)>0$ for each non-empty $A\in\mathcal{F}_{T}$, and
the sigma-field $\mathcal{F}_{0}$ has a single atom, that is, $\mathcal{F}%
_{0}=\{\emptyset,\Omega\}$. We refer to the atoms of~$\mathcal{F}_{t}$ as the
\textbf{nodes} at time $t=0,\ldots,T$, and write~$\Lambda_{t}$ for the set of
nodes at time $t=0,\ldots,T$. For any non-terminal node $\lambda\in\Lambda
_{t}$, where $t=0,\ldots,T-1$, we denote by $\mathrm{succ}(\lambda)$ the set
of \textbf{successor nodes} of~$\lambda$, that is, nodes $\mu\in\Lambda_{t+1}$
such that $\mu\subset\lambda$.

For each $t=0,\ldots,T$, we can identify any $\mathcal{F}_{t}$-measurable
random variable~$X$ with a function on~$\Lambda_{t}$, and will
write~$X^{\lambda}$ for the value of~$X$ at a node $\lambda\in\Lambda_{t}$.

We shall say that
\[
Q:=\{Q_{t}^{\lambda}\,|\,t=0,\ldots,T-1,\lambda\in\Lambda_{t}\}
\]
is a \textbf{family of absolutely continuous single-step probability measures}
whenever~$Q_{t}^{\lambda}$ is a probability measure defined on the
sigma-field
\[
\lambda\cap\mathcal{F}_{t+1}:=\left\{  \lambda\cap A\,|\,A\in\mathcal{F}%
_{t+1}\right\}
\]
for each $t=0,\ldots,T-1$ and $\lambda\in\Lambda_{t}$. Note that absolute
continuity of these measures with respect to~$\mathbb{P}$ is automatically
ensured by the assumption that $\mathbb{P}(A)>0$ for each non-empty
$A\in\mathcal{F}_{T}$. Such a family of measures gives rise to a unique
probability measure$~\mathbb{Q}$ defined on\ the sigma-field$~\mathcal{F}_{T}$
by%
\begin{equation}
\mathbb{Q}(\lambda):=\prod_{t=0}^{T-1}Q_{t}^{\lambda_{t}}(\lambda_{t+1})
\label{Eq:hfn74ba9}%
\end{equation}
for any $\lambda\in\Lambda_{T}$, where $\lambda_{t}\in\Lambda_{t}$ for
$t=0,\ldots,T$ is the unique sequence of nodes such that $\lambda_{0}%
\supset\cdots\supset\lambda_{T}=\lambda$. In general, the family $Q$ is a
richer object than the corresponding measure~$\mathbb{Q}$ in that it carries
more information at those nodes~$\lambda$ where $\mathbb{Q}(\lambda)=0$.

Furthermore, we shall say that an adapted process $S$ is a \textbf{martingale}
with respect to the family of measures~$Q$ if%
\begin{equation}
S_{t}^{\lambda}=\sum_{\mu\in\mathrm{succ}(\lambda)}Q_{t}^{\lambda}(\mu
)S_{t+1}^{\mu} \label{Eq:nfdmm0am2}%
\end{equation}
for each $t=0,\ldots,T-1$ and $\lambda\in\Lambda_{t}$. This condition implies
that, in particular, $S$~is a martingale (in the usual sense) under the
probability measure~$\mathbb{Q}$ related to the family~$Q$
by~(\ref{Eq:hfn74ba9}).

Families of absolutely continuous single-step probability measures and
martingales with respect to such families of measures will be used to
characterise the absence of arbitrage in a market model with combined (fixed
and proportional) transaction costs; see Theorem~\ref{Thm:0sjfd67fbd}.

\section{Model with fixed and proportional costs}

Let $A$, $B$ and $C$ be $\mathbb{R}$-valued processes adapted to the
filtration $\mathcal{F}$ such that $0<B\leq A<\infty$ and $0<C<\infty$. We
refer to this collection of processes together with the filtration as a
\textbf{combined-cost model}, in which $A,B$ play the respective roles of ask
and bid stock prices, with $C$ representing fixed transaction costs.

The notions of solvency and self-financing can be formalised as follows in the
combined-cost model.

\begin{definition}
\label{Def:hf6gsc}\upshape

\begin{enumerate}
\item[$1)$] We say that a portfolio $(x,y)\in\mathbb{R}^{2}$ of cash and stock
is \textbf{combined-cost solvent} at time $t=0,\ldots,T$ and node $\lambda
\in\Lambda_{t}$ when liquidating the stock position leaves a non-negative cash
amount
\[
x+B_{t}^{\lambda}y^{+}-A_{t}^{\lambda}y^{-}-C_{t}^{\lambda}\geq0
\]
after the fixed transaction cost~$C_{t}^{\lambda}$ is met, or when both the
cash and stock positions are non-negative to begin with, that is,
\[
x,y\geq0.
\]
We denote by~$\mathcal{G}_{t}^{\lambda}$\ the set of such portfolios $(x,y)$.

\item[$2)$] We define a \textbf{combined-cost self-financing strategy} as an
$\mathbb{R}^{2}$-valued $\mathcal{F}$-predictable process $(X,Y)=\{(X_{t}%
,Y_{t})\}_{t=0}^{T+1}$ such that%
\[
(X_{t}^{\lambda}-X_{t+1}^{\lambda},Y_{t}^{\lambda}-Y_{t+1}^{\lambda}%
)\in\mathcal{G}_{t}^{\lambda}%
\]
for each $t=0,\ldots,T$ and $\lambda\in\Lambda_{t}$.
\end{enumerate}
\end{definition}

\begin{remark}
\label{Rem:jf745bga}\upshape We can also consider the \textbf{combined-cost
liquidation value}%
\[
L_{t}^{\lambda}(x,y):=x+(B_{t}^{\lambda}y^{+}-A_{t}^{\lambda}y^{-}%
-C_{t}^{\lambda})\mathbf{1}_{y\notin\lbrack0,C_{t}^{\lambda}/B_{t}^{\lambda}]}%
\]
of a portfolio $(x,y)\in\mathbb{R}^{2}$ at time $t=0,\ldots,T$ and node
$\lambda\in\Lambda_{t}$. Observe that $(x,y)\in\mathcal{G}_{t}^{\lambda}$ is
equivalent to $L_{t}^{\lambda}(x,y)\geq0$. Figure~\ref{fig:anf64nag} shows a
typical set $\mathcal{G}_{t}^{\lambda}$ of combined-cost solvent portfolios.
\end{remark}

\begin{figure}[t]
\begin{center}
\begin{tikzpicture}[scale=1.0]
\draw[draw=none,fill=lightgray,opacity=0.5] (-1.5,2.75)--(0,0.5)--(0,0)--(1,0)--(3,-1)--(3,2.75)--cycle;
\draw[black,thick] (-1.5,2.75)--(0,0.5)--(0,0)--(1,0)--(3,-1);
\draw[black] (1.2,1.4)node[anchor=west] {$\mathcal{G}^\lambda_t$};
\draw [->] (-2,0)--(3,0) node (yaxis) [above left] {$y$};
\draw [->] (0,-1.5)--(0,2.75) node (yaxis) [below right] {$x$};
\draw (1,0)node[anchor=north] {$2$};
\draw (0,0.5)node[anchor=east] {$1$};
\draw (0,0)node[anchor=north east] {$0$};
\end{tikzpicture}
\end{center}
\caption{Set $\mathcal{G}^{\lambda}_{t}$ of combined-cost solvent portfolios
$(x,y)$ with $A^{\lambda}_{t}=1.5$, $B^{\lambda}_{t}=0.5$, $C^{\lambda}%
_{t}=1.0$.}%
\label{fig:anf64nag}%
\end{figure}
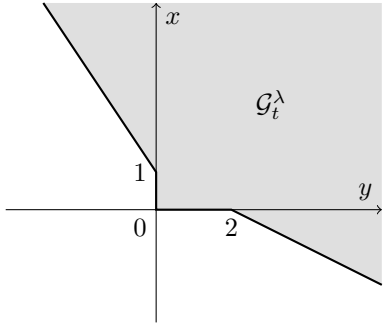

\section{Fundamental Theorem of Asset Pricing under fixed and proportional
costs\label{Sect:bf7dtd}}

\begin{definition}
\label{Def:mfs6eehn}\upshape We say that a combined-cost self-financing
strategy $(X,Y)$ is a \textbf{combined-cost arbitrage opportunity} whenever
the following conditions hold:

\begin{enumerate}
\item[$1)$] $(X_{0},Y_{0})=(0,0)$,

\item[$2)$] $X_{T+1}\geq0$ and $Y_{T+1}\geq0$,

\item[$3)$] $X_{T+1}^{\lambda}>0$ for some $\lambda\in\Lambda_{T}$.
\end{enumerate}
\end{definition}

\begin{remark}
\upshape The absence of combined-cost arbitrage can be described in terms of
the liquidation value introduced in Remark~\ref{Rem:jf745bga}. Namely, there
is no combined-cost arbitrage opportunity if and only if $L_{T}(X_{T}%
,Y_{T})=0$ for every combined-cost self-financing strategy $(X,Y)$ such that
$(X_{0},Y_{0})=(0,0)$ and $L_{T}(X_{T},Y_{T})\geq0$. This is a direct
extension of the classical no-arbitrage (NA) condition as in \cite{KabStr2001}%
, \cite{Sch2004} and others.
\end{remark}

The Fundamental Theorem of Asset Pricing extends to the case of a
combined-cost model as follows.

\begin{theorem}
\label{Thm:0sjfd67fbd}The following conditions are equivalent:

\begin{enumerate}
\item[$1)$] There is no combined-cost arbitrage opportunity in the model with
ask and bid prices $A,B$ and fixed costs~$C$;

\item[$2)$] There exist an adapted process~$S$ and a family of absolutely
continuous single-step probability measures~$Q$ such that $S$ is a martingale
with respect to~the family~$Q$ and $B\leq S\leq A$.
\end{enumerate}
\end{theorem}

\begin{proof}
To prove the implication $1)\Rightarrow2)$, assume that there is no
combined-cost arbitrage opportunity. We begin by constructing two adapted
processes $U$ and~$V$ by backward induction:%
\[
U_{T}^{\lambda}:=A_{T}^{\lambda},\quad V_{T}^{\lambda}:=B_{T}^{\lambda}%
\]
for each $\lambda\in\Lambda_{T}$, and%
\[
U_{t-1}^{\lambda}:=\max_{\mu\in\mathrm{succ}(\lambda)}(U_{t}^{\mu}\wedge
A_{t}^{\mu}),\quad V_{t-1}^{\lambda}:=\min_{\mu\in\mathrm{succ}(\lambda
)}(V_{t}^{\mu}\vee B_{t}^{\mu})
\]
for each $t=1,\ldots,T$ and $\lambda\in\Lambda_{t-1}$.

Having constructed the processes $U$ and~$V$, we claim that for each
$t=0,\ldots,T-1$ there exist stopping times $\sigma,\tau>t$ such that%
\[
U_{t}\geq A_{\sigma},\quad V_{t}\leq B_{\tau}.
\]
We prove the existence of~$\sigma$ by backward induction. For~$\tau$ the
argument is similar and will be omitted for brevity. For $t=T-1$ we get
$U_{T-1}\geq A_{\sigma}$ by putting $\sigma:=T$. Now suppose that for some
$t=1,\ldots,T-1$ we have already established that there is a stopping time
$\eta>t$ such that $U_{t}\geq A_{\eta}$. Let us put%
\[
\sigma:=\eta\mathbf{1}_{\left\{  A_{t}>U_{t}\right\}  }+t\mathbf{1}_{\left\{
A_{t}\leq U_{t}\right\}  }.
\]
It follows that%
\[
U_{t-1}\geq U_{t}\wedge A_{t}=U_{t}\mathbf{1}_{\left\{  A_{t}>U_{t}\right\}
}+A_{t}\mathbf{1}_{\left\{  A_{t}\leq U_{t}\right\}  }\geq A_{\eta}%
\mathbf{1}_{\left\{  A_{t}>U_{t}\right\}  }+A_{t}\mathbf{1}_{\left\{
A_{t}\leq U_{t}\right\}  }=A_{\sigma},
\]
completing the proof of the claim.

Next we show that
\begin{equation}
V_{t}\vee B_{t}\leq U_{t}\wedge A_{t}\label{Eq:nfd64ma0}%
\end{equation}
for each $t=0,...,T$. Suppose that this were not so, and take the largest
$t=0,...,T$ such that (\ref{Eq:nfd64ma0}) is violated. Since $U_{T}=A_{T}$ and
$V_{T}=B_{T}$, it follows that $t<T$. It also follows that $V_{t+1}\vee
B_{t+1}\leq U_{t+1}\wedge A_{t+1}$, which implies that $V_{t}\leq U_{t}$.
Moreover, we know that $B_{t}\leq A_{t}$. Hence, for (\ref{Eq:nfd64ma0}) to be
violated, at least one of the following two inequalities would have to hold at
some node $\lambda\in\Lambda_{t}$:

\begin{itemize}
\item $V_{t}^{\lambda}>A_{t}^{\lambda}$. We know that there is a stopping time
$\tau>t$ such that $B_{\tau}\geq V_{t}$, so $B_{\tau}>A_{t}$ on~$\lambda$. In
this case the strategy to buy a large enough position in stock for~$A_{t}%
^{\lambda}$ at time~$t$ and node~$\lambda$, and to sell it for~$B_{\tau}$ at
time~$\tau$ for any scenario belonging to~$\lambda$ (and otherwise to do
nothing) would be a combined-cost arbitrage opportunity. To be precise, such a
strategy $(X,Y)$ could be defined as%
\[
(X_{s},Y_{s}):=\left\{
\begin{array}
[c]{ll}%
(0,0) & \text{for }s=0,\ldots,t,\\
\mathbf{1}_{\lambda}(-A_{t}z-C_{t},z) & \text{for }s=t+1,\ldots,\tau,\\
\mathbf{1}_{\lambda}(-A_{t}z+B_{\tau}z-C_{t}-C_{\tau},0) & \text{for }%
s=\tau+1,\ldots,T+1,
\end{array}
\right.
\]
for a large enough $z>0$ so that $\left(  -A_{t}+B_{\tau}\right)
z>C_{t}+C_{\tau}$ on~$\lambda$.

\item $U_{t}^{\lambda}<B_{t}^{\lambda}$. We know that there is a stopping time
$\sigma>t$ such that $A_{\sigma}\leq U_{t}$, so $A_{\sigma}<B_{t}$
on~$\lambda$. The strategy $(X,Y)$ defined as%
\[
(X_{s},Y_{s}):=\left\{
\begin{array}
[c]{ll}%
(0,0) & \text{for }s=0,\ldots,t,\\
\mathbf{1}_{\lambda}(B_{t}z-C_{t},-z) & \text{for }s=t+1,\ldots,\sigma,\\
\mathbf{1}_{\lambda}(B_{t}z-A_{\sigma}z-C_{t}-C_{\sigma},0) & \text{for
}s=\sigma+1,\ldots,T+1,
\end{array}
\right.
\]
would be a combined-cost arbitrage opportunity when $z>0$ is large enough so
that $\left(  B_{t}-A_{\sigma}\right)  z>C_{t}+C_{\sigma}$ on~$\lambda$.
\end{itemize}

\noindent This contradicts the assumption that there is no combined-cost
arbitrage opportunity. Claim (\ref{Eq:nfd64ma0}) has therefore been proved.

We are ready to construct a process~$S$ and a family of single-step
probability measures~$Q$ by induction. At time $t=0$ we take any value%
\[
S_{0}\in\lbrack V_{0}\vee B_{0},U_{0}\wedge A_{0}].
\]
Now suppose that an $\mathcal{F}_{t}$-measurable random variable $S_{t}%
\in\lbrack V_{t}\vee B_{t},U_{t}\wedge A_{t}]$ has already been constructed
for some $t=0,\ldots,T-1$. For each $\lambda\in\Lambda_{t}$ we have%
\[
V_{t+1}^{\nu}\vee B_{t+1}^{\nu}= V_{t}^{\lambda}\leq S_{t}^{\lambda}\leq
U_{t}^{\lambda}= U_{t+1}^{\mu}\wedge A_{t+1}^{\mu}%
\]
for some $\mu,\nu\in\mathrm{succ}(\lambda)$. If $\mu\neq\nu$, we put%
\[
S_{t+1}^{\mu}:=U_{t+1}^{\mu}\wedge A_{t+1}^{\mu},\quad S_{t+1}^{\nu}%
:=V_{t+1}^{\nu}\vee B_{t+1}^{\nu}%
\]
and, for any $\eta\in\mathrm{succ}(\lambda)$ other than $\mu$ or~$\nu$, we
take as~$S_{t+1}^{\eta}$ any value%
\[
S_{t+1}^{\eta}\in\lbrack V_{t+1}^{\eta}\vee B_{t+1}^{\eta},U_{t+1}^{\eta
}\wedge A_{t+1}^{\eta}].
\]
This means that
\[
\min_{\mu\in\mathrm{succ}(\lambda)}S_{t+1}^{\mu}\leq S_{t}^{\lambda}\leq
\max_{\mu\in\mathrm{succ}(\lambda)}S_{t+1}^{\mu},
\]
so there is a probability measure~$Q_{t}^{\lambda}$ on the
sigma-field~$\lambda\cap\mathcal{F}_{t+1}$ such that (\ref{Eq:nfdmm0am2})
holds. But if $\mu=\nu$, then we put%
\[
S_{t+1}^{\mu}:=S_{0}%
\]
and, for any $\eta\in\mathrm{succ}(\lambda)$ other than~$\mu$, we take
as~$S_{t+1}^{\eta}$ any value%
\[
S_{t+1}^{\eta}\in\lbrack V_{t+1}^{\eta}\vee B_{t+1}^{\eta},U_{t+1}^{\eta
}\wedge A_{t+1}^{\eta}].
\]
Moreover, we put $Q_{t}^{\lambda}(\mu):=1$ and $Q_{t}^{\lambda}(\eta):=0$ for
any $\eta\in\mathrm{succ}(\lambda)$ other than~$\mu$, which defines a
probability measure~$Q_{t}^{\lambda}$ on the sigma-field~$\lambda
\cap\mathcal{F}_{t+1}$ such that (\ref{Eq:nfdmm0am2}) holds. This construction
produces an adapted process\ $S$ such that $B\leq S\leq A$, and a family of
absolutely continuous single-step probability measures~$Q$ such that $S$ is a
martingale with respect to~the family~$Q$. The implication $1)\Rightarrow2)$
has been proved.

Conversely, to verify that $2)\Rightarrow1)$, we assume that condition~$2)$
holds and, to argue by \emph{reductio ad absurdum}, suppose that there is a
combined-cost arbitrage opportunity $(X,Y)$. Condition~$2)$ implies that%
\begin{equation}
X_{t}^{\lambda}+S_{t-1}^{\lambda}Y_{t}^{\lambda}\geq\min_{\mu\in
\mathrm{succ}(\lambda)}(X_{t}^{\lambda}+S_{t}^{\mu}Y_{t}^{\lambda})
\label{Eq:o9fja7man1}%
\end{equation}
for each $t=1,\ldots,T$ and $\lambda\in\Lambda_{t-1}$. Indeed, if this
inequality failed for some $t=1,\ldots,T$ and $\lambda\in\Lambda_{t-1}$, then
we would have%
\[
S_{t}^{\lambda}<\min_{\mu\in\mathrm{succ}(\lambda)}S_{t+1}^{\mu}\quad
\text{or}\quad S_{t}^{\lambda}>\max_{\mu\in\mathrm{succ}(\lambda)}S_{t+1}%
^{\mu},
\]
depending on whether $Y_{t}^{\lambda}>0$ or $Y_{t}^{\lambda}<0$. In either
case it would mean that there is no probability measure~$Q_{t}^{\lambda}$ on
the sigma-field~$\lambda\cap\mathcal{F}_{t+1}$\ such that (\ref{Eq:nfdmm0am2})
holds, violating condition~$2)$ of the theorem. Next, since $(X,Y)$ is a
combined-cost self-financing strategy, it follows that, for each
$t=0,\ldots,T$ and $\lambda\in\Lambda_{t}$,%
\begin{align}
&  X_{t}^{\lambda}-X_{t+1}^{\lambda}+S_{t}^{\lambda}(Y_{t}^{\lambda}%
-Y_{t+1}^{\lambda})\nonumber\\
&  \quad\quad\quad\geq X_{t}^{\lambda}-X_{t+1}^{\lambda}+B_{t}^{\lambda}%
(Y_{t}^{\lambda}-Y_{t+1}^{\lambda})^{+}-A_{t}^{\lambda}(Y_{t}^{\lambda
}-Y_{t+1}^{\lambda})^{-}\nonumber\\
&  \quad\quad\quad\geq C_{t}^{\lambda}>0 \label{Eq:nfyatp0m}%
\end{align}
or%
\begin{equation}
X_{t}^{\lambda}\geq X_{t+1}^{\lambda},\quad Y_{t}^{\lambda}\geq Y_{t+1}%
^{\lambda}. \label{Eq:yrteramtad}%
\end{equation}
Hence%
\begin{equation}
X_{t}+S_{t}Y_{t}\geq X_{t+1}+S_{t}Y_{t+1} \label{Eq:lfmf90as74n1}%
\end{equation}
for each $t=0,\ldots,T$. We can show by backward induction that%
\begin{equation}
X_{t+1}+S_{t}Y_{t+1}\geq0 \label{Eq:nf6agem1}%
\end{equation}
for each $t=0,\ldots,T$. Clearly, (\ref{Eq:nf6agem1}) holds for $t=T$, given
that\ $X_{T+1}\geq0$ and $Y_{T+1}\geq0$. Now suppose that~(\ref{Eq:nf6agem1})
holds for some $t=1,\ldots,T$. Take any $\lambda\in\Lambda_{t-1}$. Then%
\[
X_{t}^{\lambda}+S_{t-1}^{\lambda}Y_{t}^{\lambda}\geq\min_{\mu\in
\mathrm{succ}(\lambda)}(X_{t}^{\lambda}+S_{t}^{\mu}Y_{t}^{\lambda})\geq
\min_{\mu\in\mathrm{succ}(\lambda)}(X_{t+1}^{\lambda}+S_{t}^{\mu}%
Y_{t+1}^{\lambda})\geq0,
\]
where the first inequality holds by~(\ref{Eq:o9fja7man1}), the second
by~(\ref{Eq:lfmf90as74n1}) and the last one by the induction hypothesis,
completing the backward induction argument.

To proceed further, let us put%
\[
t:=\max\{s=0,\ldots,T+1\,|\,X_{0}\geq\cdots\geq X_{s}\text{ and }Y_{0}%
\geq\cdots\geq Y_{s}\}.
\]
Since $(X,Y)$ is a combined-cost arbitrage opportunity, we know that
$X_{0}=Y_{0}=0$ and $X_{T+1}^{\lambda}>0$ for some $\lambda\in\Lambda_{T}$. If
$t=T+1$, it would mean that $0=X_{0}\geq X_{T+1}^{\lambda}>0$, a
contradiction. On the other hand, if $t\leq T$, then there would be a
$\lambda\in\Lambda_{t}$ such that (\ref{Eq:yrteramtad}) fails, so
(\ref{Eq:nfyatp0m}) would have to hold, implying that%
\[
X_{t}^{\lambda}+S_{t}^{\lambda}Y_{t}^{\lambda}>X_{t+1}^{\lambda}%
+S_{t}^{\lambda}Y_{t+1}^{\lambda}\geq0,
\]
where the last inequality follows from~(\ref{Eq:nf6agem1}). However, that too
is impossible as $X_{t}\leq X_{0}=0$ and $Y_{t}\leq Y_{0}=0$. This
contradiction completes the proof.
\end{proof}

\section{Fixed costs}

A \textbf{fixed-cost model} involves two processes $S$ and $C$ adapted to the
filtration~$\mathcal{F}$, where $0<S<\infty$ represents the stock prices and
$0<C<\infty$ the fixed transaction costs. This is a special case of the
combined-cost model when the ask and bid prices coincide. Hence, fixed-cost
solvent portfolios, fixed-cost self-financing strategies and fixed-cost
arbitrage opportunities are covered by Definitions~\ref{Def:hf6gsc}
and~\ref{Def:mfs6eehn} with $A:=B:=S$. In this case
Theorem~\ref{Thm:0sjfd67fbd} reduces to the following result.

\begin{corollary}
\label{Cor:bd64ma0n}The following conditions are equivalent:

\begin{enumerate}
\item[$1)$] There is no fixed-cost arbitrage opportunity in the model with
stock prices~$S$ and fixed costs~$C$;

\item[$2)$] There exists a family of absolutely continuous single-step
probability measures~$Q$ such that $S$ is a martingale with respect to~$Q$.
\end{enumerate}
\end{corollary}

This version of the Fundamental Theorem of Asset Pricing under fixed costs is
similar to that obtained by Jouini, Kallal and Napp \cite{JouKalNap2001}.
However, our method of proof (the proof of Theorem~\ref{Thm:0sjfd67fbd}) is
different. Moreover, the equivalent condition for the lack of fixed-cost
arbitrage is expressed in terms of single-step measures only, whereas that in
\cite{JouKalNap2001} relies on a larger family of measures.

As a consequence of Theorem~\ref{Thm:0sjfd67fbd}, together with
Corollary~\ref{Cor:bd64ma0n}, we also obtain an alternative characterisation
of the lack of combined-cost arbitrage in terms of an embedded arbitrage-free
fixed-cost model. It resembles earlier results for proportional transaction
costs, which involve embedding an arbitrage-free frictionless model; for
example, see\ Roux \cite{Rou2011}.

\begin{corollary}
\label{Cor:f8f56sgfsmac}The following conditions are equivalent:

\begin{enumerate}
\item[$1)$] There is no combined-cost arbitrage opportunity in the model with
ask and bid prices $A,B$ and fixed costs~$C$;

\item[$2)$] There exists a process~$S$ adapted to the filtration~$\mathcal{F}$
such that $B\leq S\leq A$ and the model with stock prices $S$ and fixed
costs$~C$ admits no fixed-cost arbitrage opportunity.
\end{enumerate}
\end{corollary}

\section{Equivalent separating probability measures}

In the context of the present paper, we also refer to the recent article by
Lepinette and Tran \cite{LepTra2017}, where the absence of asymptotic
arbitrage in a class of non-convex models (which include models with fixed and
proportional costs) is characterised by the existence of an \textbf{equivalent
separating probability measure} (ESPM). By definition, an ESPM is a
probability measure~$\mathbb{Q}$ equivalent to~$\mathbb{P}$ such that%
\begin{equation}
\mathbb{E}_{\mathbb{Q}}[L_{T}(X_{T+1},Y_{T+1})]\leq0 \label{Eq:0tue7wnam}%
\end{equation}
for all self-financing strategies $(X,Y)$ starting with initial endowment
$(X_{0},Y_{0})=(0,0)$, where $L_{T}(X_{T+1},Y_{T+1})$ is the liquidation value
(see Remark~\ref{Rem:jf745bga}) of the terminal portfolio $(X_{T+1},Y_{T+1})$.
The following example shows that it is possible for a model to be free of
combined-cost (or fixed-cost) arbitrage opportunities though there is no ESPM.

We conclude that ESPM's are unsuitable to characterise the lack of
combined-cost arbitrage. Indeed, the absence of combined-cost arbitrage
opportunities is consistent with the no-arbitrage (NA) condition in
Definition~5.1 of \cite{LepTra2017}. Until now, no suitable analogue of risk
neutral probabilities has been put forward to characterise this kind of NA
condition. This is resolved by our Theorem~\ref{Thm:0sjfd67fbd}, according to
which families of absolutely continuous single-step martingale measures can
play this role.

\begin{example}
\label{Exl:nf7am0ax}\upshape Consider the single-step model with stock prices
$S=A=B$ and fixed costs $C$ such that%
\[%
\begin{array}
[c]{lll}
&  & S_{1}^{\mathrm{u}}=2,C_{1}^{\mathrm{u}}=1\\
& \nearrow & \\
S_{0}=1,C_{0}=1 & \rightarrow & S_{1}^{\mathrm{d}}=1,C_{1}^{\mathrm{d}}=1
\end{array}
\]
By Theorem~\ref{Thm:0sjfd67fbd}, this model does not admit a combined-costs
(or indeed fixed-cost) arbitrage opportunity, as there is a family~$Q$ of
absolutely continuous single-step measures such that $S$ is a martingale with
respect that family; the family consists of just one measure~$Q_{0}$ such that
$Q_{0}(\mathrm{u})=1$ and $Q_{0}(\mathrm{d})=0$. Observe that
\[%
\begin{array}
[c]{lllll}
&  &  &  & (X_{2}^{\mathrm{u}},Y_{2}^{\mathrm{u}})=(-y-1,y)\\
&  &  & \nearrow & \\
(X_{0},Y_{0})=(0,0) & \rightarrow & (X_{1},Y_{1})=(-y-1,y) & \rightarrow &
(X_{2}^{\mathrm{d}},Y_{2}^{\mathrm{d}})=(-y-1,y)
\end{array}
\]
is a combined-cost self-financing strategy for any $y>1$. Suppose that
$\mathbb{Q}$ is an ESPM, so we have $0<\mathbb{Q}(\mathrm{u})<1$. Hence%
\[
\mathbb{E}_{\mathbb{Q}}[L_{1}(X_{2},Y_{2})]=\mathbb{Q}(\mathrm{u})\left(
y-2\right)  +(1-\mathbb{Q}(\mathrm{u}))\left(  -2\right)  =qy-2
\]
when $y>0$. Taking $y>\frac{2}{\mathbb{Q}(\mathrm{u})}$, we get $\mathbb{E}%
_{\mathbb{Q}}[L_{1}(X_{2},Y_{2})]>0$, contradicting~(\ref{Eq:0tue7wnam}).
\end{example}

\section{Concluding remarks}

In this work the classical Fundamental Theorem of Asset Pricing due to
Harrison and Pliska \cite{HarPli1981} is extended to discrete market models
with simultaneous fixed and proportional transaction costs and finite state
space. This also extends later work on the Fundamental Theorem of Asset
Pricing under proportional costs such as \cite{JouKal1995}, \cite{KabStr2001},
\cite{Ort2001}, \cite{Rou2011}, and under fixed costs \cite{JouKalNap2001}.

Developments for models with infinite state space and/or continuous time
and/or several assets are likely to follow. Moreover, as the Fundamental
Theorem of Asset Pricing has now been established for markets with
simultaneous fixed and proportional costs, it will inform research on pricing
and hedging derivative securities in this setting.

\end{document}